\documentclass[12pt, a4paper]{article}
\usepackage{latexsym}
\usepackage{xy}
\usepackage{amsfonts,amsmath,amssymb,amsthm}
\usepackage{color}

\xyoption{all}

\newcommand{\bcen}{\begin{center}}     \newcommand{\ecen}{\end{center}}
\newcommand{\bay}{\begin{array}}      \newcommand{\eay}{\end{array}}
\newcommand{\beq}{\begin{eqnarray*}}      \newcommand{\eeq}{\end{eqnarray*}}

\def\ca{{\cal A}}
\def\cb{{\cal B}}
\def\cc{{\cal C}}
\def\cd{{\cal D}}
\def\ce{{\cal E}}
\def\cf{{\cal F}}

\def\Fun{\mathrm{Fun}}
\def\gl{\mathrm{gl.dim}}

\def\Hom{\mathrm{Hom}}

\def\mod{\mathrm{mod}}
\def\Mod{\mathrm{Mod}}

\def\op{\mathrm{op}}

\def\proj{\mathrm{proj}}

\def\rad{\mathrm{rad}}

\def\rep{\mathrm{rep}}
\def\Rep{\mathrm{Rep}}

\begin{document}

\newtheorem{theorem}{Theorem}
\newtheorem{proposition}{Proposition}
\newtheorem{lemma}{Lemma}
\newtheorem{corollary}{Corollary}
\newtheorem{remark}{Remark}
\newtheorem{example}{Example}
\newtheorem{definition}{Definition}
\newtheorem*{conjecture}{Conjecture}
\newtheorem{question}{Question}

\title{\large\bf A construction of dualizing categories by tensor products of categories }

\author{\large Yang Han and Ningmei Zhang}

\date{\footnotesize
KLMM, ISS, AMSS, Chinese Academy of Sciences, Beijing 100190, P.R.
China.\\ School of Mathematical Science, University of
Chinese Academy of Sciences. \\
E-mail: hany@iss.ac.cn}

\maketitle

\begin{abstract} It is shown that the idempotent completion of the
additive hull of the tensor product of the residue category of the
category of paths of a locally finite quiver modulo an admissible
ideal and a dualizing category is dualizing. Furthermore, the
category of finitely presented functors over such tensor product
category is dualizing and has almost split sequences. As
applications, the categories of all kinds of complexes are proved to
have almost split sequences.
\end{abstract}

\medskip


\medskip


\bigskip

\section{\large Introduction}

Throughout this paper we assume that $k$ is a commutative artin ring
unless stated otherwise. Denote by $\Mod k$ the category of
$k$-modules and $\mod k$ the full subcategory of $\Mod k$ consisting
of all finitely generated $k$-modules. Let $E$ be an injective
envelope of the factor module $k/\rad k$ of $k$ modulo its radical
$\rad k$ in $\Mod k$, and $D := \Hom_k(-,E)$. A {\it dualizing
$k$-category} or {\it dualizing $k$-variety} $\ca$ is a Hom-finite
Krull-Schmidt $k$-category such that the duality $D :
\Fun_k(\ca^\op, \mod k) \rightarrow \Fun_k(\ca, \mod k)$ induces a
duality $D : \mod \ca \rightarrow \mod \ca^\op$. Dualizing
$k$-categories were introduced by Auslander and Reiten as a
generalization of artin $k$-algebras \cite{AusRei74}. It is
well-known that the existence of almost split sequences is quite
useful in the representation theory of algebras (Ref. \cite[Chapter
2]{Rei82}). A $k$-category ${\cal A}$ being dualizing ensures that
the category $ {\rm mod} {\cal A}$ of finitely presented functors in
${\rm Mod}{\cal A}$ has almost split sequences (Ref. \cite[Theorem
7.1.3]{Rei82}). From a given dualizing $k$-category $\ca$, there are
some known constructions of dualizing $k$-categories such as
$\mod\ca$ (Ref. \cite[Proposition 2.6]{AusRei74}), the functorially
finite Krull-Schmidt full $k$-subcategories of $\ca$ (Ref.
\cite[Theorem 2.3]{AusSma81} and \cite[\S 9.7 Example 5]{GabRoi92}),
the residue categories ${\cal A}/(1_A)$ of ${\cal A}$ modulo the
ideal $(1_A)$ of ${\cal A}$ generated by the identity morphism $1_A$
of an object $A$ in ${\cal A}$ (Ref. \cite[\S 9.7, Example
8]{GabRoi92}), and the category $C^b(\mod\ca)$ of bounded complexes
over $\mod\ca$ (Ref. \cite[Theorem 4.3]{BauSouZua05}).

In this paper, we will give another construction of dualizing
$k$-categories by tensor products of $k$-categories which can be
applied to construct a large number of new dualizing $k$-categories
from a given dualizing $k$-category. Let $Q$ be a locally finite
quiver, $kQ$ the $k$-category of paths of $Q$, $I$ an admissible
ideal of $kQ$ generated by a set of paths in $Q$, $\cb := kQ/I$ the
residue category of $kQ$ modulo $I$, and $\ca$ a dualizing
$k$-category. We will prove that the idempotent completion
$|\oplus(\cb \otimes_k \ca)$ of the additive hull $\oplus(\cb
\otimes_k \ca)$ of the tensor product $\cb \otimes_k \ca$ of
$k$-categories $\cb$ and $\ca$, is a dualizing $k$-category.
Furthermore, we will show that $\mod(\cb \otimes_k \ca)$ is a
dualizing $k$-category and has almost split sequences. This is a
natural generalization of \cite[Proposition 2.6]{AusRei74},
\cite[Theorem 4.3]{BauSouZua05}, and so on.

The $n$-complexes was introduced by Mayer in 1942 for setting up a
new homology theory \cite{May42, Spa49}. This generalized homology
theory was studied in
\cite{Kap96,DubKer96,Dub98,KasWam98,CibSolWis07}. The projectives
and injectives in the category of $n$-complexes were described in
\cite{Tik02,Est07}. The homotopty category and derived category of
the category of $n$-complexes were studied in
\cite{YanDin15,Gil15,BahHafNem16,IyaKatMiy13}. The $n$-complexes
were also applied to study generalized Koszul algebras
\cite{Ber01,BerDubWam03,GreMarMarZha04}. Moreover, the categories of
complexes with amplitude in an interval play important roles in the
theory of derived representation types
\cite{BauSouZua05,Bau06,Bau07,Zha16}. As one application of our
results, we will prove that the category $C^b_n(\mod\ca)$ (resp.
$C^b_n(\ca)$) of bounded $n$-complexes over $\mod\ca$ (resp. $\ca$
with $\gl(\mod\ca)<\infty$) and the category $C^m_n(\mod\ca)$ (resp.
$C^m_n(\ca)$) of $n$-complexes over $\mod\ca$ (resp. $\ca$) with
amplitude in the interval $[1,m]$ have almost split sequences. These
generalize some results in \cite{BauSouZua05}. The $n$-cyclic
complexes or $n$-cycle complexes were introduced by Peng and Xiao
(Ref. \cite[\S 7, Appendix]{PenXia97}) which were used to realize
simple Lie algebras and their quantum enveloping algebras (Ref.
\cite{PenXia97,Bri13,CheDen15}). As the other application of our
results, we will show that the category $C_{\mathbb{Z}_n}(\mod\ca)$
(resp. $C^b_{\mathbb{Z}_n}(\ca)$, $C_{\mathbb{Z}_n}(\ca)$) of
$n$-cyclic complexes (resp. bounded $n$-cyclic complexes, $n$-cyclic
complexes) over $\mod\ca$ (resp. $\ca$ with $\gl(\mod\ca)<\infty$,
$\ca$ with $\gl(\mod\ca) \leq 1$) has almost split sequences. Note
that the almost split sequences in the category
$C_{\mathbb{Z}_n}(\proj A)$ of $n$-cyclic complexes over the
category $\proj A$ of finitely generated projective modules over a
finite dimensional hereditary algebra $A$, are described in
\cite{RinZha11,CheDen15}.

\section{Preliminaries}

In this section, we will fix some notations and terminologies on all
kinds of categories and functors, and almost split sequences.

\subsection{Categories and functors}

A category $\cc$ is said to be {\it skeletally small} if all
isomorphism classes of objects in $\cc$ form a set. Note that a
skeletally small category is called a {\it svelte} category in
\cite[2.1]{GabRoi92}. For a skeletally small category $\cc$ and a
category $\cd$, we denote by $\Fun(\cc,\cd)$ the {\it category of
functors from $\cc$ to $\cd$} whose objects are all functors from
$\cc$ to $\cd$ and whose morphisms are all natural transformations
between these functors.

A {\it $k$-category} is a category $\ca$ whose morphism sets are
endowed with $k$-module structures such that the composition maps
are $k$-bilinear (Ref. \cite[\S 2.1]{GabRoi92}). A {\it $k$-functor}
from a $k$-category $\ca$ to a $k$-category $\cb$ is a functor $F :
\ca \rightarrow \cb$ such that the defining maps $F(A,A') :
\ca(A,A') \rightarrow \cb(FA,FA')$ are $k$-linear for all $A,A' \in
\ca$. For a skeletally small $k$-category $\ca$ and a $k$-category
$\cb$, we denote by $\Fun_k(\ca,\cb)$ the {\it category of
$k$-functors from $\ca$ to $\cb$}, i.e., the full subcategory of the
functor category $\Fun(\ca,\cb)$ consisting of all $k$-functors,
which is also a $k$-category.

Let $\ca$ and $\cb$ be two $k$-categories. The {\it tensor product}
of $\ca$ and $\cb$ is the $k$-category $\ca \otimes_k \cb$ whose
objects are pairs $(A,B)$ and whose Hom sets $(\ca \otimes_k
\cb)((A,B),(A',B')) := \ca(A,A') \otimes_k \cb(B,B')$ for all $A,A'
\in \ca$ and $B,B' \in \cb$ (Ref. \cite[Page 13]{Mit72}). It is
well-known that if $\ca,\cb$ are skeletally small $k$-categories and
$\cc$ is a $k$-category then $\Fun_k(\ca,\Fun_k(\cb,\cc)) \cong
\Fun_k(\ca \otimes_k \cb,\cc)$ (Ref. \cite[Page 13]{Mit72}).

For a skeletally small $k$-category $\ca$, we denote by $\Mod\ca$
the {\it category of right $\ca$-modules}, i.e., the category
$\Fun_k(\ca^\op,\Mod k)$ of $k$-functors from the opposite category
$\ca^\op$ of $\ca$ to $\Mod k$. Clearly, $\Mod\ca$ is an abelian
category. A functor $M \in \Mod \ca$ is {\it representable} if $M
\cong \ca(-,A)$ for some $A \in \ca$. A functor $M \in \Mod \ca$ is
{\it finitely generated} if there is an epimorphism $\oplus_{i \in
I}\ca(-,A_i) \twoheadrightarrow M$ for a finite index set $I$ and
$A_i \in \ca$. A functor $M \in \Mod \ca$ is {\it finitely
presented} if there is an exact sequence $\oplus_{j \in
J}\ca(-,A'_j) \rightarrow \oplus_{i \in I}\ca(-,A_i)
\twoheadrightarrow M$ for two finite index sets $I,J$ and $A_i,A'_j
\in \ca$. Note that once $\ca$ is a skeletally small additive
$k$-category then a functor $M \in \Mod \ca$ is finitely generated
if and only if there is an epimorphism $\ca(-,A) \twoheadrightarrow
M$ for some $A \in \ca$, and a functor $M \in \Mod \ca$ is finitely
presented if and only if there is an exact sequence $\ca(-,A')
\rightarrow \ca(-,A) \twoheadrightarrow M$ for some $A,A' \in \ca$.
Denote by $\mod \ca$ the full subcategory of $\Mod \ca$ consisting
of all finitely presented functors (Ref. \cite[Page 22]{GabRoi92}).
It is well-known that $\mod\ca$ is abelian if and only if $\ca$ has
pseudo-kernels, i.e., for any morphism $f \in \ca(A',A)$ there is a
morphism $f' \in \ca(A'',A')$ such that $\xymatrix{ \ca(-,A'')
\ar[r]^-{\ca(-,f')} & \ca(-,A') \ar[r]^-{\ca(-,f)} & \ca(-,A) }$ is
exact (Ref. \cite[Page 102, Proposition]{Aus71}).

\subsection{Dualizing categories and almost split sequences}

A $k$-category $\ca$ is said to be {\it Hom-finite} if all Hom sets
$\ca(A,A')$ are finitely generated $k$-modules for all $A,A'$ in
$\ca$. A skeletally small additive $k$-category $\ca$ is said to be
{\it Krull-Schmit} if each object in $\ca$ is a finite direct sum of
indecomposables with local endomorphism algebras. Note that a
Krull-Schmidt category is called a {\it multilocular category} in
\cite[3.1]{GabRoi92}. A skeletally small Hom-finite additive
$k$-category $\ca$ is Krull-Schmidt if and only if all idempotents
in $\ca$ split (Ref. \cite[Theorem 3.3]{GabRoi92}), i.e., for each
idempotent $e \in \ca(A,A)$ there are $A' \in \ca$, $f \in
\ca(A,A')$ and $g \in \ca(A',A)$ such that $e=gf$ and $fg=1_{A'}$.

A {\it dualizing $k$-category} or {\it dualizing $k$-variety} $\ca$
is a Hom-finite Krull-Schmidt $k$-category such that the natural
duality $D : \Fun_k(\ca^\op, \mod k) \rightarrow \Fun_k(\ca, \mod
k), F \mapsto DF$, where $(DF)(A):= D(F(A))$ and $(DF)(f):= D(F(f))$
for all $A \in \ca$ and $f \in \ca(A,A')$, induces a duality $D :
\mod \ca \rightarrow \mod \ca^\op$. Dualizing $k$-categories were
introduced by Auslander and Reiten as a generalization of artin
$k$-algebras (Ref. \cite{AusRei74}). For an artin $k$-algebra
$\Lambda$, the category $\proj \Lambda$ of finitely generated
projective $\Lambda$-modules and the category $\mod \Lambda$ of
finitely generated $\Lambda$-modules are dualizing $k$-categories
(Ref. \cite[Proposition 2.5 and Proposition 2.6]{AusRei74}). A {\it
locally bounded $k$-category} is a Hom-finite Krull-Schmidt
$k$-category $\ca$ satisfying that for each $A$ in $\ca$ there are
only finitely many isomorphism classes of indecomposable objects
$A'$ in $\ca$ with $\ca(A',A) \neq 0$ or $\ca(A,A') \neq 0$. This is
the additivization of the locally bounded $k$-category in \cite[\S
2.1]{BonGab82}. Every locally bounded $k$-category is a dualizing
$k$-category (Ref. \cite[Proposition 7.1.5]{Rei82}).

Let $\ca$ be a Krull-Schmidt $k$-category. A {\it right almost split
morphism} in $\ca$ is a morphism $g : Y \rightarrow Z$ in $\ca$
which is not a retraction and for any non-retraction $g' : Y'
\rightarrow Z$ there is a morphism $g'' : Y' \rightarrow Y$ such
that $g'=gg''$. We say that $\ca$ {\it has right almost split
morphisms} if for any indecomposable $Z$ in $\ca$ there is a right
almost split morphism ending in $Z$. A {\it left almost split
morphism} in $\ca$ is a morphism $f : X \rightarrow Y$ in $\ca$
which is not a section and for any non-section $f' : X \rightarrow
Y'$ there is a morphism $f'' : Y \rightarrow Y'$ such that
$f'=f''f$. We say that $\ca$ {\it has left almost split morphisms}
if for any indecomposable $X$ in $\ca$ there is a left almost split
morphism starting in $X$. We say that $\ca$ {\it has almost split
morphisms} if $\ca$ has both right and left almost split morphisms.

Let $\ca$ be an additive $k$-category. A pair $(i,d)$ of composable
morphisms $\xymatrix{ X \ar[r]^-i & Y \ar[r]^-d & Z }$ of $\ca$ is
said to be {\it exact} if $i$ is a kernel of $d$ and $d$ is a
cokernel of $i$. Let $\ce$ be a class of exact pairs $\xymatrix{ X
\ar[r]^-i & Y \ar[r]^-d & Z }$ which is closed under isomorphisms.
An exact pair in $\ce$ is called a {\it conflation}. The morphisms
$i$ and $d$ appearing in a conflation $(i,d)$ are called an {\it
inflation} and a {\it deflation} respectively. The class $\ce$ is
called an {\it exact structure} on $\ca$ and $(\ca,\ce)$ is called
an {\it exact $k$-category} (Ref. \cite[\S 9.1]{GabRoi92} and
\cite[\S 1.1 and Appendix]{DraReiSmaSol99}) if the following axioms
are satisfied:

(E1) The composition of two deflations is a deflation.

(E2) For each $f \in \ca(Z',Z)$ and each deflation $d \in \ca(Y,Z)$,
there is $Y' \in \ca$, $f' \in \ca(Y',Y)$ and a deflation $d' \in
\ca(Y',Z')$ such that $df'=fd'$.

(E3) Identities are deflations. If $de$ is a deflation, then so is
$d$.

(E3$^\op$) Identities are inflations. If $ji$ is an inflation, then
so is $i$.

It is well-known that if $(\ca,\ce)$ is a skeletally small exact
$k$-category then $(\ca^\op,\ce^\op)$ with $\ce^\op :=
\{(d^\op,i^\op) \; | \; (i,d)\in \ce \}$ is an exact $k$-category as
well (Ref. \cite[\S 9.1 Example 2]{GabRoi92}). Note also that an
abelian category admits a natural exact structure whose conflations
are all short exact sequences.

Let $(\ca,\ce)$ be a Krull-Schmidt exact $k$-category. An {\it
almost split sequence} in $(\ca,\ce)$ is a conflation $\xymatrix{ X
\ar[r]^-f & Y \ar[r]^-g & Z}$ in $\ce$ with $f$ a left almost split
morphism and $g$ a right almost split morphism (Ref. \cite[\S
2.2]{DraReiSmaSol99}). We say that $(\ca,\ce)$ {\it has almost split
sequence} if $\ca$ has almost split morphisms, for any
indecomposable non-$\ce$-projective object $Z$ there is an almost
split sequence ending in $Z$, and for any indecomposable
non-$\ce$-injective object $A$ there is an almost split sequence
starting in $X$. It is well-known that if $\ca$ is a dualizing
$k$-category then $\mod\ca$ has almost split sequences (Ref.
\cite[Theorem 7.1.3]{Rei82}).

Let $\ca$ be an additive $k$-category. A full additive $k$-subcategory $\cc$
of $\ca$ is said to be {\it contravariantly finite} in $\ca$ if for
each $A \in \ca$ the restriction $\ca(-,A)|\cc$ of $\ca(-,A)$ to
$\cc$ is a finitely generated functor on $\cc$, i.e., there is an
epimorphism $\cc(-,C) \twoheadrightarrow \ca(-,A)|\cc$ for some $C
\in \cc$. Equivalently, $\cc$ is contravariantly finite in $\ca$ if for each
$A \in \ca$, there is a morphism $C \rightarrow A$ (called the {\it
right $\cc$-approximation} of $A$ in $\cc$) with $C \in \cc$
such that $\cc(C',C) \rightarrow \ca(C',A) \rightarrow 0$ is exact
for all $C' \in \cc$. Dually, a full additive $k$-subcategory $\cc$ of $\ca$
is said to be {\it covariantly finite} in $\ca$ if for each $A \in
\ca$ the restriction $\ca(A,-)|\cc$ of $\ca(A,-)$ to $\cc$ is a
finitely generated functor on $\cc$, i.e., there is an epimorphism
$\cc(C,-) \twoheadrightarrow \ca(A,-)|\cc$ for some $C \in \cc$.
Equivalently, $\cc$ is contravariantly finite in $\ca$ if for each
$A \in \ca$, there is a morphism $C \rightarrow A$ (called the {\it
left $\cc$-approximation} of $A$ in $\cc$) with $C \in \cc$ such
that $\cc(C,C') \rightarrow \ca(A,C') \rightarrow 0$ is exact for
all $C' \in \cc$. Furthermore, a $k$-subcategory $\cc$ of $\ca$ is
said to be {\it functorially finite} in $\ca$ if it is both
contravariantly and covariantly finite in $\ca$. These definition
were introduced by Auslander and Smal{\o} in \cite[\S 3, Page
81]{AusSma80}. If a Krull-Schmidt exact $k$-category $(\ca , \ce)$
has almost split sequences, and $\cc$ is a functorially finite full
$k$-subcategory of $\ca$ closed under conflations and direct
summands, then $(\cc,\ce|\cc)$ has almost split sequences (Ref.
\cite[Theorem 2.4]{AusSma81}).

\section{A construction of dualizing categories}

In this section, we will provide a new construction of dualizing
categories from a given dualizing category by idempotent
completions, additive hulls, and tensor products of $k$-categories.
The representations of quivers in $k$-linear categories will play a
key role in the construction as well.

\subsection{Idempotent completions and additive hulls}

The {\it idempotent completion} of a $k$-category $\ca$ is the
$k$-category $|\ca$ whose objects are all pairs $(A,e)$ with $A \in
\ca$ and $e \in \ca(A,A)$ being idempotent, and whose Hom sets are
$(|\ca)((A,e),(A',e')) := e'\ca(A,A')e$ (Ref. \cite[\S 2.1, Example
7]{GabRoi92}). Clearly, all idempotents in $|\ca$ split, and if
$\ca$ is skeletally small (resp. Hom-finite, additive) then so is
$|\ca$. Moreover, $(|\ca)^\op \cong |(\ca^\op)$.

The {\it additive hull} of a $k$-category $\ca$ is the $k$-category
$\oplus\ca$ whose objects are all $m$-tuples $(A_1, \cdots , A_m)$
with $m \in \mathbb{N}$ and $A_i \in \ca$, and whose Hom sets are
$(\oplus\ca)((A_1, \cdots , A_m), (A'_1, \cdots , A'_n)) :=
(\ca(A_i,A'_j))_{j,i}$ (Ref. \cite[\S 2.1, Example 8]{GabRoi92}).
Clearly, $\oplus\ca$ is additive, and if $\ca$ is skeletally small
(resp. Hom-finite) then so is $\oplus\ca$. Moreover,
$(\oplus\ca)^\op \cong \oplus(\ca^\op)$.

\begin{lemma} \label{Lemma-|+KrullSchmidt}
Let $\ca$ be a skeletally small Hom-finite $k$-category.
Then the idempotent completion $|\oplus\ca$ of the additive hull
$\oplus\ca$ of $\ca$ is a Krull-Schmidt $k$-category.
\end{lemma}

\begin{proof} Obviously, $|\oplus\ca$ is a skeletally small Hom-finite additive
$k$-category. Moreover, all idempotents in $|\oplus\ca$ split. It
follows from \cite[Theorem 3.3]{GabRoi92} that $|\oplus\ca$ is
Krull-Schmidt.
\end{proof}

Let $\ca$ be a skeletally small $k$-category. Then the functor $F :
\ca \rightarrow |\ca , A \mapsto (A,1_A)$, is fully faithful, which
induces an equivalence $F_* : \Mod(|\ca) \linebreak \rightarrow
\Mod\ca , M \mapsto MF$ (Ref. \cite[\S 2.2, Example 6]{GabRoi92}).
Moreover, the functor $G : \ca \rightarrow \oplus\ca , A \mapsto
(A)$, is also fully faithful, which induces an equivalence $G_* :
\Mod(\oplus\ca) \rightarrow \Mod\ca , N \mapsto NG$ (Ref. \cite[\S
2.2, Example 7]{GabRoi92}). Restricted to the categories of finitely
presented functors, we have equivalences $\mod(|\ca) \simeq \mod\ca
\simeq \mod(\oplus\ca)$, and further $\mod(|\oplus\ca) \simeq
\mod\ca$.

\subsection{Representations of a quiver in a category}

Let $Q$ be a quiver, i.e., a directed graph. Denote by $Q_0$ the set
of vertices of $Q$ and by $Q_1$ the set of arrows of $Q$. For an
arrow $a \in Q_1$, denote by $s(a)$ and $t(a)$ the source and target
of $a$ respectively. A {\it path} $p$ of length $l \geq 1$ with
source $s(p)=v$ and target $t(p)=w$ is a sequence $a_l \cdots
a_2a_1$ of arrows $a_i$ such that $s(a_1)=v, s(a_{i+1})=t(a_i)$ for
all $1 \leq i \leq l-1$, and $t(a_l)=w$. Besides the paths of length
$\geq 1$, there are also {\it trivial paths} $1_v$ of length $0$
with both source and target $v$ for all $v \in Q_0$. A quiver $Q$
gives rise to the {\it category of paths of $Q$}, denoted by $Q$ as
well, whose objects are vertices of $Q$, whose Hom sets $Q(v,w)$
consists of the paths with source $v$ and target $w$, and the
composition is the juxtaposition of paths. Usually, we view a quiver
as a category
--- its category of paths. A quiver $Q$ also gives rise to the {\it
$k$-category of paths of $Q$}, denoted by $kQ$, whose objects are
vertices of $Q$ and whose Hom sets $(kQ)(v,w)$ are the $k$-vector
spaces with a basis $Q(v,w)$ (Ref. \cite[\S 2.1 Example
6]{GabRoi92}). The {\it opposite quiver} of a quiver $Q$ is the
quiver $Q^\op$ with $(Q^\op)_0 := Q_0$ and $(Q^\op)_1 := \{a^\op \;
| \; a \in Q_1, s(a^\op) := t(a), t(a^\op) := s(a) \}$. Obviously,
the ($k$-)category of paths of $Q^\op$ is isomorphic to the opposite
category of the ($k$-)category of paths of $Q$.

To a quiver $Q$, we associate a quiver $P(Q)^l$ called the {\it left
path space of $Q$}, whose vertices are the paths $p$ of $Q$ and
whose arrows are pairs $(p,ap) : p \rightarrow ap$ with $a \in Q_1$
satisfying $t(p)=s(a)$. For a vertex $v$ of $Q$, denote by
$P(Q)^l_v$ the connected component of $P(Q)^l$ that $1_v$ lies in.

A {\it representation} of a quiver $Q$ in a category $\ca$ is a
functor $R \in \Fun(Q,\ca)$. Obviously, a representation $R$ of $Q$
in $\ca$ is determined by assigning an object $R(v) \in \ca$ to each
vertex $v \in Q_0$ and a morphism $R(a) \in \ca(R(s(a)),R(t(a)))$ to
each arrow $a \in Q_1$. A {\it morphism} $\phi$ between two
representations $R$ and $R'$ is a natural transformation, i.e., a
family of morphisms $\phi_v \in \ca(R(v),R'(v))$ with $v \in Q_0$
such that $R'(a) \circ \phi_{s(a)} = \phi_{t(a)} \circ R(a)$, i.e.,
the following diagram is commutative:
$$\xymatrix{ R(s(a)) \ar[r]^-{\phi_{s(a)}} \ar[d]^-{R(a)} & R'(s(a))
\ar[d]^-{R'(a)} \\ R(t(a)) \ar[r]^-{\phi_{t(a)}} & R'(t(a)) }$$ for
all arrows $a \in Q_1$. Obviously, the category $\Rep(Q, \ca)$ of
representations of a quiver $Q$ in a $k$-category $\ca$, i.e, the
functor category $\Fun(Q,\ca)$, is isomorphic to the $k$-functor
category $\Fun_k(kQ,\ca)$.

Let $I$ be an {\it ideal} of $Q$ consisting of some paths of length
at least 2, i.e., a set of paths of length at least 2 closed under
left or right concatenation with any concatenatable path of $Q$, or
equivalently, a set of paths of length at least 2 such that all the
paths of $Q$ containing a path in $I$ as a subpath are in $I$. We
define $Q_I$, called {\it quiver $Q$ with monomial relations $I$},
to be the category whose objects are all vertices of $Q$ and whose
morphisms are all the paths of $Q$ not in $I$. Clearly, the category
$\Rep(Q_I,\ca) := \Fun(Q_I,\ca)$ of representations of $Q_I$ in a
pre-additive category $\ca$ is isomorphic to the full subcategory of
$\Rep(Q,\ca)$ consisting of all representations $R$ in $\Rep(Q,\ca)$
such that $R(p)=0$ for all $p \in I$. By abuse of terminology, we
still denote by $I$ the ideal of $kQ$ generated by all paths in $I$,
which will not cause any confusion. Then the category
$\Rep(Q_I,\ca)$ of representations of $Q_I$ in a $k$-category $\ca$
is isomorphic to $\Fun_k(kQ/I,\ca)$.

For a vertex $v \in Q_0$, we define $P(Q_I)^l_v = \{p \in P(Q)^l_v
\; | \; p \notin I\}$. Denote by $\mathbb{A}_1$ the quiver having
just one vertex $\bullet$ and no arrows. We define functors $f_v :
\mathbb{A}_1 \rightarrow Q_I, \bullet \mapsto v$, $g_v :
\mathbb{A}_1 \rightarrow P(Q_I)^l_v, \bullet \mapsto 1_v$, and $t_v
: P(Q_I)^l_v \rightarrow Q_I$ as follows: for a vertex $p$ of
$P(Q_I)^l_v$, $t_v(p) := t(p)$; for an arrow $(p,ap)$ of
$P(Q_I)^l_v$, $t_v(p,ap) := a$. Thus $f_v=t_v \circ g_v$. Note that
each functor $h$ from a skeletally small category $\cb$ to a
skeletally small category $\cc$ induces a functor $h_* : \Fun(\cc,
\ca) \rightarrow \Fun(\cb, \ca), F \mapsto F \circ h$. So we have
induced functors $f_{v*} : \Rep(Q_I,\ca) \rightarrow
\Rep(\mathbb{A}_1,\ca)$, $g_{v*} : \Rep(P(Q_I)^l_v,\ca) \rightarrow
\Rep(\mathbb{A}_1,\ca)$ and $t_{v*} : \Rep(Q_I,\ca) \rightarrow
\Rep(P(Q_I)^l_v, \ca)$. Moreover, $f_{v*} = (t_v \circ g_v)_* =
g_{v*} \circ t_{v*}$.

Let $\ca$ be a cocomplete category, i.e., small coproducts exist in
$\ca$. Completely analogous to \cite[Proposition 3.1 and
3.2]{EnoOyoTor04}, the functors $g_{v*}$ and $t_{v*}$ have left
adjoints $g^*_v : \Rep(\mathbb{A}_1,\ca) \rightarrow
\Rep(P(Q_I)^l_v,\ca)$ and $t^*_v : \Rep(P(Q_I)^l_v,\ca) \linebreak
\rightarrow \Rep(Q_I,\ca)$ respectively. Thus the functor $f_{v*}$
has a left adjoint $f^*_v := t^*_v \circ g^*_v$. The functor $g^*_v
: \Rep(\mathbb{A}_1,\ca) \rightarrow \Rep(P(Q_I)^l_v,\ca)$ is
defined as follows: for a representation $A$ in
$\Rep(\mathbb{A}_1,\ca)$, or equivalently, an object $A \in \ca$, we
define $g^*_v(A)$ to be the representation of $P(Q_I)^l_v$ which
sends each vertex to $A$, and each arrow to $1_A$; for a morphism in
$\Rep(\mathbb{A}_1,\ca)$, or equivalently, a morphism $f : A
\rightarrow A'$ in $\ca$, we define $g^*_v(f)_p = f$ for all vertex
$p$ of $P(Q_I)^l_v$. The left adjoint $t^*_v$ of $t_{v*}$ is defined
as follows: for a given representation $M$ of $P(Q_I)^l_v$, we
define $t^*_v(M)$ as follows: for a vertex $w$ of $Q$, $t^*_v(M)(w)
:= \oplus_{p \in Q_I(v,w)}M(p)$ where $M(p)$ is the object in $\ca$
corresponding to the vertex $p$ of $P(Q_I)^l_v$ under the
representation $M$ of $P(Q_I)^l_v$; for an arrow $a \in Q_1$, the
morphism $t^*_v(M)(a) : t^*_v(M)(s(a)) \rightarrow t^*_v(M)(t(a))$
is defined as
$$(\alpha_{qp})_{q,p} : \oplus_{p
\in Q_I(v,s(a))} M(p) \rightarrow \oplus_{q \in Q_I(v,t(a))}M(q)$$
where $\alpha_{qp} := M(p,ap) : M(p) \rightarrow M(ap)$ is the
morphism in $\ca$ corresponding to the arrow $(p,ap) : p \rightarrow
ap$ of $P(Q_I)^l_v$ under the representation $M$ of $P(Q_I)^l_v$ if
$q=ap$, and $\alpha_{qp} := 0$ otherwise. Moreover, we need to
define $t^*_v$ for morphisms. If $f : M \rightarrow M'$ is a
morphism in $\Rep(P(Q_I)^l_v, \ca)$ then, for each vertex $p$ of
$P(Q_I)^l_v$, we have a morphism $f_p \in \ca(M(p),M'(p))$. For each
vertex $w$ of $Q$, we define the morphism $t^*_v(f)_w := \oplus_{p
\in Q_I(v,w)}f_p \in \ca(\oplus_{p \in Q_I(v,w)}M(p), \oplus_{p \in
Q_I(v,w)}M'(p))$.

\begin{lemma} \label{Lemma-Projective} {\rm
(See \cite[Theorem 3.3]{EnoOyoTor04} and \cite[Page 3217]{Est07})}
Let $Q$ be a quiver, $I$ a closed set of monomial relations of $Q$,
and $\ca$ a cocomplete abelian category having enough projective
objects. Then $\Rep(Q_I,\ca)$ is a cocomplete abelian category
having enough projective objects and $\{f^*_v(P) \; | \; v \in Q_0,
\; P \in \ca \; \mbox{projecive} \}$ is a family of projective
generators for $\Rep(Q_I, \ca)$.
\end{lemma}

\begin{proof} Clearly, $\Rep(Q_I,\ca)$ is a cocomplete abelian
category. Now we show that it has enough projective objects. We have
known that there are adjoint isomorphisms
$$(\Rep(Q_I,\ca))(f^*_v(A),R) \cong
(\Rep(\mathbb{A}_1,\ca))(A,f_{v*}(R))$$ for all $R \in
\Rep(Q_I,\ca)$ and $A \in \Rep(\mathbb{A}_1,\ca) = \ca$. Note that
we always identify the category $\Rep(\mathbb{A}_1,\ca)$ with $\ca$.
For any $R \in \Rep(Q_I, \ca)$, we have $f_{v*}(R) \in \ca =
\Rep(\mathbb{A}_1,\ca)$. Since $\ca$ has enough projective objects,
there is a projective object $P_v \in \ca$ and an epimorphism $d_v :
P_v \twoheadrightarrow f_{v*}(R)$. By the adjoint isomorphism, there
is a unique morphism $\phi^v : f^*_v(P_v) \rightarrow R$ such that
$(\phi^v)_v :(f^*_v(P_v))(v) \rightarrow R(v)$ is just $d_v : P_v
\rightarrow f_{v*}(R)$. Repeat this procedure for every vertex $v$
of $Q_I$, we get morphisms of representations $\phi^v : f^*_v(P_v)
\rightarrow R$ for all $v$ of $Q_I$. By definition, $f_{v*}$
preserves exactness. Thus $f^*_v$ preserves projectiveness. Hence
$f^*_v(P_v)$ is a projective representation of $Q_I$. So is
$\oplus_{v \in Q_0}f^*_v(P_v)$. Furthermore, the morphism of
representations $\phi=(\phi^v)_{v \in Q_0} : \oplus_{v \in Q_0}
f^*_v(P_v) \rightarrow R$ is an epimorphism, since the restriction
$(\phi^v)_v=d_v$ of $\phi_v$ on the $v$-component $f^*_v(P_v)(v)$ of
$(\oplus_{v \in Q_0} f^*_v(P_v))(v)$ is. It follows that $\{f^*_v(P)
\; | \; v \in Q_0, P \in \ca \; \mbox{projecive} \}$ is a family of
projective generators for $\Rep(Q_I, \ca)$.
\end{proof}

\subsection{A construction of dualizing categories}

A quiver $Q$ is said to be {\it locally finite} if for each vertex
$v$ of $Q$, there are only finitely many arrows of $Q$ with source
or target $v$ (\cite[\S 8.3, Example 2]{GabRoi92}). Let $Q$ be a
quiver, $kQ$ the $k$-category of paths of $Q$, $k^1Q$ the ideal of
$kQ$ generated by all arrows of $Q$, and $k^rQ := (k^1Q)^r$ for all
$r \in \mathbb{N}$. An ideal $I$ of $kQ$ is said to be {\it
admissible} if $I \subseteq k^2Q$ and for each vertex $v$ of $Q$
there is an $l_v \in \mathbb{N}$ such that $I$ contains all paths of
length $\geq l_v$ with source or target $v$ (Ref. \cite[\S 8.3,
Example 2]{GabRoi92})).

{\it From now on, we always assume that $I$ is an admissible ideal
of $kQ$ generated by a set of paths in $Q$.} So all paths of $Q$ in
$I$ form an ideal $J$ of $Q$ consisting of some paths of length at
least 2. Once we denote by $Q_I$ the category whose objects are
vertices of $Q$ and whose morphisms are all paths of $Q$ that are
not in $I$, then $Q_I = Q_J$. Moreover, the opposite category
$(kQ/I)^\op$ is a residue category of $(kQ)^\op=kQ^\op$ modulo an
admissible ideal $I^\op$ of $kQ^\op$ generated by a set of paths in
$Q^\op$. Indeed, all paths of $Q^\op$ in $I^\op$ form an ideal
$J^\op = \{ a^\op_1a^\op_2 \cdots a^\op_l \; | \; a_l \cdots a_2a_1
\in J \}$ of $Q^\op$. Let $\ca$ be an additive $k$-category. We
denote by $\rep(kQ/I,\ca)$ (resp. $\rep(Q_I,\ca)$) the full
subcategory of $\Rep(kQ/I,\ca)$ (resp. $\Rep(Q_I,\ca)$) consisting
of all {\it support-finite} representations, i.e., the
representations $R$ satisfying $R(v) \neq 0$ for only finitely many
vertices $v$ of $Q$. Then $\rep(kQ/I,\ca) \simeq \rep(Q_I,\ca)$.

\begin{proposition} \label{Proposition-rep-mod}
Let $Q$ be a locally finite quiver, $I$ an admissible ideal of $kQ$
generated by a set of paths of $Q$, $\cb:=kQ/I$ and $\ca$ a
skeletally small additive $k$-category with pseudo-kernels. Then the
equivalence $\Rep(\cb,\Mod\ca) \simeq \Mod(\cb^\op \otimes_k \ca)$
restricts to an equivalence $\rep(\cb,\mod\ca) \simeq \mod(\cb^\op
\otimes_k \ca)$.

$$\xymatrix{ \Rep(\cb,\Mod\ca) \ar[r]^-{\simeq} & \Mod(\cb^\op \otimes_k \ca) \\
\rep(\cb,\mod\ca) \ar@{^{(}->}[u] \ar[r]^-{\simeq} & \mod(\cb^\op
\otimes_k \ca) \ar@{^{(}->}[u] }$$
\end{proposition}

\begin{proof} First of all, we have natural equivalences $$\begin{array}{rcl} \Rep(\cb, \Mod\ca)
& = & \Fun_k(\cb, \Fun_k(\ca^\op, \Mod k)) \\ & \cong & \Fun_k(\cb
\otimes_k \ca^\op, \Mod k) \\ & \cong & \Fun_k((\cb^\op \otimes_k
\ca)^\op, \Mod k) \\ & = & \Mod(\cb^\op \otimes_k \ca).
\end{array}$$ The equivalence functor $\Phi : \Rep(\cb, \Mod\ca)
\rightarrow \Mod(\cb^\op \otimes_k \ca)$ is given by $\Phi(R)((v,A))
:= R(v)(A)$ and $\Phi(R)(a \otimes f) := R(a)_A \circ R(s(a))(f) =
R(t(a))(f) \circ R(a)_{A'} : R(s(a))(A') \rightarrow R(t(a))(A)$
$$\xymatrix{ R(s(a))(A) \ar[r]^-{R(a)_A} & R(t(a))(A) \\
R(s(a))(A') \ar[r]^-{R(a)_{A'}} \ar[u]_-{R(s(a))(f)} & R(t(a))(A')
\ar[u]_-{R(t(a))(f)} }$$ for all arrows $a \in Q_1$ and morphisms $f
: A \rightarrow A'$ in $\ca$. The quasi-inverse $\Psi : \Mod(\cb^\op
\otimes_k \ca) \rightarrow \Rep(\cb, \Mod\ca)$ of $\Phi$ is given by
$\Psi(M)(v)(A) := M((v,A))$ and $\Psi(M)(v)(f) := M(1_v \otimes f)$
for all $v \in Q_0, A \in \ca, f \in \ca(A,A')$, $\Psi(M)(a)_A :=
M(a \otimes 1_A)$ for all $a \in Q_1, A \in \ca$, and $(\Psi(F)_v)_A
:= F_{(v,A)} : \Psi(M)(v)(A) = M((v,A)) \rightarrow \Psi(M')(v)(A) =
M'((v,A))$ for all $v \in Q_0, A \in \ca, F \in (\Mod(\cb^\op
\otimes_k \ca))(M,M')$.

Since $\Mod\ca$ has enough projective objects $\ca(-,A)$ with $A \in
\ca$, by Lemma~\ref{Lemma-Projective}, $\Rep(\cb,\Mod\ca) \simeq
\Rep(Q_I,\Mod\ca)$ has enough projective objects $\{P_{v,A} :=
f^*_v(\ca(-,A)) \; | \; v \in Q_0, A \in \ca \}$. Note that $Q$ is
locally finite and $I$ is admissible, by the definition of $f^*_v$,
we have $P_{v,A} \in \rep(Q_I,\mod\ca)$. Since $\ca$ has
pseudo-kernels, $\mod\ca$ is abelian \cite[Page 102,
Proposition]{Aus71}. So is $\rep(\cb,\mod\ca) \simeq
\rep(Q_I,\mod\ca)$. For each $R \in \rep(\cb,\mod\ca)$, by the proof
of Lemma~\ref{Lemma-Projective}, there is an exact sequence
$$\oplus_{v \in Q_0,A \in |\ca|} P_{v,A}^{m_{v,A}} \rightarrow
\oplus_{v \in Q_0,A \in |\ca|} P_{v,A}^{n_{v,A}} \twoheadrightarrow
R$$ where $|\ca|$ is a complete set of representatives of
isomorphism classes of objects in $\ca$, and the nonnegative
integers $m_{v,A}$ and $n_{v,A}$ are nonzero for only finitely many
pairs $(v,A) \in Q_0 \times |\ca|$. It is not difficult to check
that $\Phi(P_{v,A}) = \cb(v,-) \otimes_k \ca(-,A) = \cb^\op(-,v)
\otimes_k \ca(-,A) = (\cb^\op \otimes_k \ca)(-,(v,A))$. Applying the
equivalence functor $\Phi$ to the above exact sequence, we obtain an
exact sequence
$$\begin{array}{lll} \oplus_{v \in Q_0,A \in |\ca|} (\cb^\op \otimes_k
\ca)(-,(v,A))^{m_{v,A}} & \rightarrow & \\ \oplus_{v \in Q_0,A \in
|\ca|} (\cb^\op \otimes_k \ca)(-,(v,A))^{n_{v,A}} &
\twoheadrightarrow & \Phi(R). \end{array}$$ Namely, $\Phi(R) \in
\mod(\cb^\op \otimes_k \ca)$.

Conversely, for any $M \in \mod(\cb^\op \otimes_k \ca)$, there is an
exact sequence $$\begin{array}{lll} \oplus_{v \in Q_0,A \in |\ca|}
(\cb^\op \otimes_k \ca)(-,(v,A))^{m_{v,A}} & \rightarrow & \\
\oplus_{v \in Q_0,A \in |\ca|} (\cb^\op \otimes_k
\ca)(-,(v,A))^{n_{v,A}} & \twoheadrightarrow & M  \end{array}$$
where the nonnegative integers $m_{v,A}$ and $n_{v,A}$ are nonzero
for only finitely many pairs $(v,A)$. Thus for each $v' \in Q_0$,
there is an exact sequence
$$\begin{array}{lll} \oplus_{v \in Q_0, A \in
|\ca|} (\cb^\op \otimes_k \ca)((v',-),(v,A))^{m_{v,A}} & \rightarrow & \\
\oplus_{v \in Q_0, A \in |\ca|} (\cb^\op \otimes_k
\ca)((v',-),(v,A))^{n_{v,A}}  & \twoheadrightarrow &
M(v',-)=\Psi(M)(v'),
\end{array}$$ i.e.,
$$\begin{array}{lll} \oplus_{v \in Q_0, A \in |\ca|} (\cb^\op(v',v) \otimes_k
\ca(-,A))^{m_{v,A}} & \rightarrow & \\ \oplus_{v \in Q_0, A \in
|\ca|} (\cb^\op(v',v) \otimes_k \ca(-,A))^{n_{v,A}} &
\twoheadrightarrow & M(v',-)=\Psi(M)(v'). \end{array}$$ Since $Q$ is
locally finite and $I$ is admissible and generated by a set of
paths, all $\cb(v,v') = \cb^\op(v',v)$ are finitely generated free
$k$-modules and there are only finitely many $v' \in Q_0$ such that
$\cb(v,v') = \cb^\op(v',v) \neq 0$. Hence $\Psi(M) \in
\rep(\cb,\mod\ca)$.
\end{proof}

\begin{remark}{\rm A sequence $\xymatrix{ 0 \ar[r] & R \ar[r]^-f &
R' \ar[r]^-{f'} & R'' \ar[r] & 0 }$ in the functor category
$\Rep(kQ/I,\Mod\ca)$ is exact if $\xymatrix{ 0 \ar[r] & R(v)
\ar[r]^-{f_v} & R'(v) \ar[r]^-{f'_v} & R''(v) \ar[r] & 0 }$ is exact
in $\Mod\ca$ for all $v \in Q_0$. It gives a natural exact structure
$\cf$ on the abelian category $\Rep(kQ/I,\Mod\ca)$ such that the
natural equivalence $\Rep(kQ/I,\Mod\ca) \simeq \Mod((kQ/I)^\op
\otimes_k \ca)$ preserves exactness. }\end{remark}

Our main result in this paper is the following theorem:

\begin{theorem} \label{Theorem-dualizing}
Let $Q$ be a locally finite quiver, $I$ an admissible ideal of $kQ$
generated by a set of paths in $Q$, $\cb:=kQ/I$ and $\ca$ a
dualizing $k$-category. Then the following statements hold:

{\rm (1)} The $k$-category $|\oplus(\cb \otimes_k \ca)$ is
dualizing.

{\rm (2)} The abelian $k$-category $\mod(\cb \otimes_k \ca)$ is
dualizing.

{\rm (3)} The abelian $k$-category $\mod(\cb \otimes_k \ca)$ has
almost split sequences.

\end{theorem}

\begin{proof} (1) By assumption, $\cb$ is a skeletally small
Hom-finite $k$-category. So is $\cb \otimes_k \ca$. It follows from
Lemma~\ref{Lemma-|+KrullSchmidt} that $|\oplus(\cb \otimes_k \ca)$
is Hom-finite and Krull-Schmidt. So we only need to show that the
duality $D : \Fun_k((|\oplus(\cb \otimes_k \ca))^\op,\mod k)
\rightarrow \Fun_k(|\oplus(\cb \otimes_k \ca),\mod k)$ induces a
duality $D : \mod(|\oplus(\cb \otimes_k \ca)) \rightarrow
\mod(|\oplus(\cb \otimes_k \ca))^\op$.

Since $(|\oplus(\cb \otimes_k \ca))^\op \cong |\oplus(\cb \otimes_k
\ca)^\op \cong |\oplus(\cb^\op \otimes_k \ca^\op)$, we have
$\mod(|\oplus(\cb \otimes_k \ca))^\op \simeq \mod(\cb^\op \otimes_k
\ca^\op)$. On the other hand, we have $\mod(|\oplus(\cb \otimes_k
\ca)) \simeq \mod(\cb \otimes_k \ca)$ as well. Thus it is enough to
prove that the duality $D$ above induces a duality $D : \mod(\cb
\otimes_k \ca) \rightarrow \mod(\cb \otimes_k \ca)^\op$.

By Proposition~\ref{Proposition-rep-mod}, we have equivalences
$\mod(\cb \otimes_k \ca) \simeq \rep(kQ^\op/I^\op, \linebreak
\mod\ca)$ and $\mod(\cb^\op \otimes_k \ca^\op) \simeq \rep(kQ/I,
\mod\ca^\op)$. Thus it suffices to prove that the duality $D$ above
induces a duality $D : \rep(kQ^\op/I^\op, \mod\ca) \rightarrow
\rep(kQ/I, \mod\ca^\op)$, or equivalently, a duality $D :
\rep(Q^\op_{I^\op}, \mod\ca) \rightarrow \rep(Q_I, \linebreak
\mod\ca^\op)$.

$$\xymatrix{ \Fun_k((|\oplus(\cb \otimes_k \ca))^\op,\mod k)
\ar[r]^-{D} & \Fun_k(|\oplus(\cb \otimes_k \ca),\mod k) \\
\mod(|\oplus(\cb \otimes_k \ca)) \ar@{^{(}->}[u] \ar@{.>}[r]^-{D}
\ar[d]^{\simeq} &
\mod(|\oplus(\cb \otimes_k \ca))^\op  \ar@{^{(}->}[u] \ar[d]^{\simeq} \\
\mod(\cb \otimes_k \ca) \ar@{.>}[r]^-{D} & \mod(\cb
\otimes_k \ca)^\op \\
\rep(kQ^\op/I^\op, \mod\ca) \ar[u]_{\simeq} \ar@{.>}[r]^-{D} &
\rep(kQ/I, \mod\ca^\op)  \ar[u]_{\simeq} \\
\rep(Q^\op_{I^\op}, \mod\ca) \ar[u]_{\simeq} \ar@{.>}[r]^-{D} &
\rep(Q_I, \mod\ca^\op)  \ar[u]_{\simeq} }$$

Since $\ca$ is dualizing, the duality $D : \Fun_k(\ca^\op, \mod k)
\rightarrow \Fun_k(\ca, \mod k)$ induces a duality $D : \mod\ca
\rightarrow \mod\ca^\op$. In the last row of the above diagram, the
functor $D : \rep(Q^\op_{I^\op}, \mod\ca) \rightarrow \rep(Q_I,
\mod\ca^\op)$ is defined as follows: For any $R \in
\rep(Q^\op_{I^\op}, \mod\ca)$, $D(R)(v)= D(R(v)) \in \mod\ca^\op$
for all $v \in Q_0$ and $D(R)(a) = D(R(a))$ for all $a \in Q_1$. For
any morphism $\phi : R \rightarrow R'$ in $\rep(Q^\op_{I^\op},
\mod\ca)$, $D(\phi) : D(R') \rightarrow D(R)$ is the morphism in
$\rep(Q_I, \mod\ca^\op)$ given by $D(\phi)_v=D(\phi_v)$ for all $v
\in Q_0$. The duality $D : \mod\ca \rightarrow \mod\ca^\op$ implies
that $D : \rep(Q^\op_{I^\op}, \mod\ca) \rightarrow \rep(Q_I,
\mod\ca^\op)$ is a duality.

(2) It follows from (1) and \cite[Proposition 2.6]{AusRei74} that
$\mod(|\oplus(\cb \otimes_k \ca))$ is an abelian dualizing
$k$-category.  Thus $\mod(\cb \otimes_k \ca) \simeq \mod(|\oplus(\cb
\otimes_k \ca))$ is also an abelian dualizing $k$-category.

(3) It follows from (1) and \cite[Theorem 7.1.3]{Rei82} that
$\mod(|\oplus(\cb \otimes_k \ca))$ has almost split sequences. Thus
$\mod(\cb \otimes_k \ca) \simeq \mod(|\oplus(\cb \otimes_k \ca))$
has almost split sequences.
\end{proof}

\begin{remark}{\rm By Theorem~\ref{Theorem-dualizing} (1),
we can construct a large number of new dualizing $k$-categories from
a given dualizing $k$-category. In practice, it is more convenient
to apply the conclusion that $\rep(kQ/I, \mod\ca)$ is a dualizing
$k$-category and has almost split sequences which can be obtained
from Proposition~\ref{Proposition-rep-mod} and
Theorem~\ref{Theorem-dualizing} (2) and (3). This is a
generalization of \cite[Theorem 4.3 and Corollary 4.4]{BauSouZua05},
which will be clear in the next section. In the case that $Q$ is
just one vertex and has no any arrows, it is nothing but
\cite[Proposition 2.6]{AusRei74}. }\end{remark}

\section{Applications}

In this section, we will apply our main theorem to show that the
categories of all kinds of complexes have almost split sequences.

\subsection{Categories of $n$-complexes}

Let $\ca$ be an additive category and $n \geq 2$. An {\it
$n$-complex} $X$ on $\ca$ is a collection $(X^i,d^i_X)_{i \in
\mathbb{Z}}$ with $X^i \in \ca$ and $d^i_X \in \ca(X^i,X^{i+1})$
such that $d^{i+n-1}_X \cdots d^{i+1}_Xd^i_X = 0$ for all $i \in
\mathbb{Z}$. An $n$-complex $X = (X^i,d^i_X)_{i \in \mathbb{Z}}$ on
$\ca$ can be visualized as the following diagram:
$$\xymatrix{\cdots \ar[r]^-{d^{i-2}_X}  & X^{i-1}
\ar[r]^-{d^{i-1}_X} & X^i \ar[r]^-{d^i_X} & X^{i+1}
\ar[r]^-{d^{i+1}_X} & \cdots . }$$ A {\it morphism} $f$ from an
$n$-complex $X = (X^i,d^i_X)_{i \in \mathbb{Z}}$ on $\ca$ to an
$n$-complex $Y = (Y^i,d^i_Y)_{i \in \mathbb{Z}}$ on $\ca$ is a
collection $(f^i)_{i \in \mathbb{Z}}$ with $f^i \in \ca(X^i,Y^i)$
such that $f^{i+1}d^i_X=d^i_Yf^i$ for all $i \in \mathbb{Z}$, i.e.,
the following diagram is commutative:
$$\xymatrix{ \cdots \ar[r]^-{d^{i-2}_X}  & X^{i-1}
\ar[r]^-{d^{i-1}_X} \ar[d]^-{f^{i-1}} & X^i \ar[r]^-{d^i_X}
\ar[d]^-{f^i} & X^{i+1} \ar[r]^-{d^{i+1}_X} \ar[d]^-{f^{i+1}} &
\cdots \\ \cdots \ar[r]^-{d^{i-2}_Y} & Y^{i-1} \ar[r]^-{d^{i-1}_Y} &
Y^i \ar[r]^-{d^i_Y} & Y^{i+1} \ar[r]^-{d^{i+1}_Y} & \cdots }$$ The
{\it composition} of morphisms $f=(f^i)_{i \in \mathbb{Z}} : X
\rightarrow Y$ and $g=(g^i)_{i \in \mathbb{Z}} : Y \rightarrow Z$ is
$gf := (g^if^i)_{i \in \mathbb{Z}} : X \rightarrow Z$. All
$n$-complexes on $\ca$ and all morphisms between them form the {\it
category of $n$-complexes on $\ca$}, denoted by $C_n(\ca)$.

Let $(\ca , \ce)$ be an exact category and $\ce_n$ the class of
composable morphisms $\xymatrix{ X \ar[r]^-{f} & Y \ar[r]^-{g} & Z
}$ in $C_n(\ca)$ such that $\xymatrix{ X^i \ar[r]^-{f^i} & Y^i
\ar[r]^-{g^i} & Z^i }$ is a conflation in $\ce$ for all $i \in
\mathbb{Z}$. Then $(C_n(\ca),\ce_n)$ is an exact category. An
$n$-complex $X = (X^i,d^i_X)_{i \in \mathbb{Z}}$ on $\ca$ is said to
be {\it bounded} if $X^i=0$ for all but finitely many $i \in
\mathbb{Z}$. Denoted by $C^b_n(\ca)$ the full subcategory of
$C_n(\ca)$ consisting of all bounded $n$-complexes on $\ca$. Let
$\ce^b_n$ be the class of the composable morphisms in both
$C^b_n(\ca)$ and $\ce_n$. Then $(C^b_n(\ca),\ce^b_n)$ is a full
exact subcategory of $(C_n(\ca),\ce_n)$.

The following result is a generalization of \cite[Theorem
4.3]{BauSouZua05}:

\begin{corollary} \label{Corollary-C^b_n-ass}
Let $\ca$ be a dualizing $k$-category, $\cf$ the natural exact
structure on the abelian category $\mod\ca$, and $n \geq 2$. Then
$(C^b_n(\mod\ca),\cf^b_n)$ has almost split sequences.
\end{corollary}

\begin{proof} Let $Q$ be
the quiver with vertices $i$ and arrows $a_i : i \rightarrow i+1$
for all $i \in \mathbb{Z}$, i.e.,
$$\unitlength=1mm \begin{picture}(70,10)
\put(0,4){$\cdots$} \multiput(18,4)(14,0){3}{$\bullet$}
\put(6,5){\vector(1,0){10}} \put(8,7){$a_{-2}$}
\put(21,5){\vector(1,0){10}} \put(22,7){$a_{-1}$}
\put(35,5){\vector(1,0){10}} \put(37,7){$a_0$}
\put(49,5){\vector(1,0){10}} \put(51,7){$a_1$} \put(60,4){$\cdots$}
\\  \put(16,0){$-1$} \put(32,0){0} \put(46,0){1} \put(68,4){,} \end{picture}$$
$I$ the ideal of $kQ$ generated by all paths of length $n$, and $\cb
:= kQ/I$. By Proposition~\ref{Proposition-rep-mod} and
Theorem~\ref{Theorem-dualizing}, we know $C^b_n(\mod\ca) \simeq
\rep(Q_I,\mod\ca) \simeq \rep(\cb,\mod\ca) \simeq \mod(\cb^\op
\otimes_k \ca)$ has almost split sequences.
\end{proof}

Let $\ca$ be a dualizing $k$-category and $\ce$ the {\it trivial
exact structure} on $\ca$, i.e., $\ce$ consists of all split short
exact sequences $\xymatrix{ X \ar[r]^-{f} & Y \ar[r]^-{g} & Z }$ in
$\ca$. Denote by $\proj\ca$ the full subcategory of $\mod\ca$
consisting of all projective $\ca$-modules, i.e., all representable
functors in $\Mod\ca$. Note that the exact category $(\ca,\ce)$ is
equivalent to the exact category $(\proj\ca,\cf | \proj\ca)$.
Indeed, $\ca \rightarrow \proj\ca, A \mapsto \ca(-,A)$, is an
equivalence functor.

For any $M \in \mod\ca$ and $j \in \mathbb{Z}$, denote by $J_j(M)$
the $n$-complex $X=(X^i,d^i_X)_{i \in \mathbb{Z}}$ where $X^i := M$
for all $i \in [j,j+n-1]$ and $X^i := 0$ otherwise, and $d^i_X :=
1_M$ for all $i \in [j,j+n-2]$ and $d^i_X := 0$ otherwise.

\begin{corollary} Let $\ca$ be a dualizing $k$-category
with trivial exact structure $\ce$ and $\gl(\mod\ca) < \infty$, and
$n \geq 2$. Then the exact category $(C^b_n(\ca),\ce^b_n)$ has
almost split sequences.
\end{corollary}

\begin{proof} First of all, we show that $(C^b_n(\ca),\ce^b_n)$
has right almost split morphisms. We have known $(\ca,\ce) \simeq
(\proj\ca,\cf | \proj\ca)$. Thus $(C^b_n(\ca),\ce^b_n) \simeq
(C^b_n(\proj\ca),(\cf | \proj\ca)^b_n)$. So it is enough to prove
that the exact category $(C^b_n(\proj\ca),(\cf | \proj\ca)^b_n) =
(C^b_n(\proj\ca),\cf^b_n | C^b_n(\proj\ca))$ has right almost split
morphisms. For this, it suffices to show that $C^b_n(\proj\ca)$ is a
contravariantly finite subcategory of $C^b_n(\mod\ca)$ closed under
$\cf^b_n$-extensions and direct summands. Obviously, we need only to
prove that $C^b_n(\proj\ca)$ is contravariantly finite over
$C^b_n(\mod\ca)$, i.e., for any $Z \in C^b_n(\mod\ca)$, there exists
a right $C^b_n(\proj\ca)$-approximation $g : Y \rightarrow Z$.

By \cite[Page 6, Formula (17)]{IyaKatMiy13}, there is an epimorphism
$p : \oplus_{j \in \mathbb{Z}}J_j(Z^j) \rightarrow Z$ in
$C^b_n(\mod\ca)$. Since $\mod\ca$ has enough projective objects, for
any $j \in \mathbb{Z}$, we have an epimorphism $P_j
\twoheadrightarrow Z^j$ in $\mod\ca$ with $P^j \in \proj\ca$ where
we take $P_j=0$ in the case of $Z^j=0$. Thus there is an epimorphism
$p' : \oplus_{j \in \mathbb{Z}}J_j(P_j) \rightarrow \oplus_{j \in
\mathbb{Z}}J_j(Z^j)$ in $C^b_n(\mod\ca)$, and further a morphism
$$\xymatrix{ r := pp' :
\oplus_{j \in \mathbb{Z}}J_j(P_j) \ar@{->>}[r] & \oplus_{j \in
\mathbb{Z}}J_j(Z^j) \ar@{->>}[r] & Z }$$ in $C^b_n(\mod\ca)$.

By the assumption $\gl(\mod\ca) < \infty$ and \cite[Proposition
41]{IyaKatMiy13}, $Z$ admits a homotopically projective resolution
$q : Q \rightarrow Z$ with $Q \in C^b_n(\proj\ca)$. Now we check
$g:=(q,r) : Y:= Q \oplus (\oplus_{j \in \mathbb{Z}}J_j(P_j))
\rightarrow Z$ is a right $C^b_n(\proj\ca)$-approximation of $Z$. In
another words, for any morphism $g' : Z' \rightarrow Z$ in
$C^b_n(\mod\ca)$ with $Z' \in C^b_n(\proj\ca)$, we need to show that
there is a morphism $g'' : Z' \rightarrow Y$ in $C^b_n(\proj\ca)$
such that $g'=gg''$.

Since $q$ is a quasi-isomorphism, $q^{-1}g' : Z' \rightarrow Q$ is a
morphism in the bounded derived category $D^b_n(\mod\ca)$. Since $Z'
\in C^b_n(\proj\ca)$ is homotopically projective, there is a
morphism $h : Z' \rightarrow Q$ in $C^b_n(\proj\ca)$ such that $h =
q^{-1}g'$ in $D^b_n(\mod\ca)$.
$$\xymatrix{ & Z' \ar[d]^-{g'} \ar[dl]_-h \\ Q \ar[r]^-q & Z }$$
Thus $g'=qh$ in $D^b_n(\mod\ca)$, and further in the bounded
homotopy category $K^b_n(\mod\ca)$ due to $Z' \in C^b_n(\proj\ca)$.
Hence $g'=qh+l$ in $C^b_n(\mod\ca)$ for some null-homotopy $l : Z'
\rightarrow Z$ in $C^b_n(\mod\ca)$. Since $l$ is a null-homotopy,
$l$ is factored through $p$ (cf. \cite[Proof of Theorem
19]{IyaKatMiy13}), say $l=pl'$. Since $Z' \in C^b_n(\proj\ca)$, each
component of $Z'$ is projective. Thus $l'$ is factored through $p'$,
say $l'=p'l''$. Hence $l$ is factored through $r$.
$$\xymatrix{ & & Z' \ar[dll]^-{l''} \ar[dl]^-{l'} \ar[d]^-l \\
\oplus_{j \in \mathbb{Z}}J_j(P_j) \ar@{->>}[r]_-{p'} & \oplus_{j \in
\mathbb{Z}}J_j(Z^j) \ar@{->>}[r]_-p & Z }$$ So we get $$g' = qh+l =
qh+rl'' = (q,r) \left(\begin{array}{c} h \\ l''
\end{array}\right) = gg''$$ where $g'' := \left(\begin{array}{c}
h \\ l'' \end{array}\right) : Z' \rightarrow Y = Q \oplus (\oplus_{j
\in \mathbb{Z}}J_j(P_j))$. Thus $g : Y \rightarrow Z$ is a right
$C^b_n(\proj\ca)$-approximation of $Z$. Therefore, $C^b_n(\proj\ca)$
is contravariantly finite over $C^b_n(\mod\ca)$.

Since $C^b_n(\ca)$ is Krull-Schmidt and $(C^b_n(\ca),\ce^b_n)$ has
right almost split morphisms, $(C^b_n(\ca),\ce^b_n)$ has minimal
right almost split morphisms. Furthermore, one can prove that if $Z
\in C^b_n(\ca)$ is indecomposable and non-$\ce^b_n$-projective then
there is an almost split sequence in $(C^b_n(\ca),\ce^b_n)$ ending
in $Z$.

Since $\ca$ is dualizing, $\mod\ca^\op \simeq (\mod\ca)^\op$. Thus
$\gl(\mod\ca^\op) = \gl(\mod\ca) < \infty$ (Ref. \cite[Page
42]{Mit72}). Applying the obtained result to dualizing $k$-variety
$\ca^\op$, we know that $(C^b_n(\ca^\op),(\ce^\op)^b_n)$ has right
almost split morphisms and if $Z \in C^b_n(\ca^\op)$ is
indecomposable and non-$(\ce^\op)^b_n$-projective then there is an
almost split sequence in $(C^b_n(\ca^\op),(\ce^\op)^b_n)$ ending in
$Z$. Since $(C^b_n(\ca),\ce^b_n) \simeq
(C^b_n(\ca^\op),(\ce^\op)^b_n)^\op$, $(C^b_n(\ca),\ce^b_n)$ has left
almost split morphisms and if $X \in C^b_n(\ca)$ is indecomposable
and non-$\ce^b_n$-injective then there is an almost split sequence
in $(C^b_n(\ca),\ce^b_n)$ starting in $X$. Therefore,
$(C^b_n(\ca),\ce^b_n)$ has almost split sequences.
\end{proof}

For $m \in \mathbb{N}$, denote by $C^m_n(\ca)$ the full subcategory
of $C_n(\ca)$ consisting of all $n$-complexes $X=(X^i,d^i_X)_{i \in
\mathbb{Z}}$ on $\ca$ with amplitude in the interval $[1,m]$, i.e.,
$X^i=0$ for all $i \notin \{1,2, \cdots, m\}$. Let $\ce^m_n$ be the
class of the composable morphisms in both $C^m_n(\ca)$ and $\ce_n$.
Then $(C^m_n(\ca),\ce^m_n)$ is a full exact subcategory of
$(C^b_n(\ca),\ce^b_n)$.

The following result is a generalization of \cite[Corollary
4.4]{BauSouZua05}:

\begin{corollary} Let $\ca$ be a dualizing $k$-category, $\cf$ the
natural exact structure on the abelian category $\mod\ca$, $m \geq
1$, and $n \geq 2$. Then $(C^m_n(\mod\ca),\cf^m_n)$ has almost split
sequences.
\end{corollary}

\begin{proof} Let $Q$ be the quiver with
vertices $i$ and arrows $a_i : i \rightarrow i+1$ for all $1 \leq i
\leq m-1$, i.e.,
$$\unitlength=1mm \begin{picture}(60,10)
\multiput(0,4)(14,0){5}{$\bullet$} \put(3,5){\vector(1,0){10}}
\put(5,7){$a_1$} \put(17,5){\vector(1,0){10}} \put(19,7){$a_2$}
\put(45,5){\vector(1,0){10}} \put(34,4){$\cdots$} \put(47,7){$a_m$}
\\  \put(0,0){1} \put(14,0){2} \put(28,0){3} \put(38,0){$m-1$}
\put(56,0){$m$} \put(64,4){,} \end{picture}$$ $I$ the ideal of $kQ$
generated by all paths of length $n$, and $\cb := kQ/I$. By
Proposition~\ref{Proposition-rep-mod} and
Theorem~\ref{Theorem-dualizing}, we know $C^m_n(\mod\ca) \simeq
\rep(Q_I,\mod\ca) \simeq \rep(\cb,\mod\ca) \simeq \mod(\cb^\op
\otimes_k \ca)$ has almost split sequences.
\end{proof}

The following result is a generalization of \cite[Theorem
4.5]{BauSouZua05}:

\begin{corollary} Let $\ca$ be a dualizing $k$-category
with trivial exact structure $\ce$, $m \geq 1$, and $n \geq 2$. Then
the exact category $(C^m_n(\ca),\ce^m_n)$ has almost split
sequences.
\end{corollary}

\begin{proof} First of all, we show that $(C^m_n(\ca),\ce^m_n)$
has right almost split morphisms. We have known $(\ca,\ce) \simeq
(\proj\ca,\cf | \proj\ca)$. Thus $(C^m_n(\ca),\ce^m_n) \simeq
(C^m_n(\proj\ca),(\cf | \proj\ca)^m_n)$. So it is enough to prove
that the exact category $(C^m_n(\proj\ca),(\cf | \proj\ca)^m_n)$ has
right almost split morphisms. For this, it suffices to show that
$C^m_n(\proj\ca)$ is a contravariantly finite subcategory of
$C^m_n(\mod\ca)$ closed under $\cf^m_n$-extensions and direct
summands. Obviously, we need only to prove that $C^m_n(\proj\ca)$ is
contravariantly finite over $C^m_n(\mod\ca)$, i.e., for any $Z \in
C^m_n(\mod\ca)$, there exists a right
$C^m_n(\proj\ca)$-approximation $g : Y \rightarrow Z$.

By \cite[Page 6, Formula (17)]{IyaKatMiy13}, there is an epimorphism
$p : \oplus^m_{j=1}J_j(Z^j) \rightarrow Z$ in $C^b_n(\mod\ca)$.
Since $\mod\ca$ has enough projective objects, for any $j \in
[1,m]$, we have an epimorphism $P_j \twoheadrightarrow Z^j$ in
$\mod\ca$ with $P^j \in \proj\ca$. Thus there is an epimorphism $p'
: \oplus^m_{j=1}J_j(P_j) \rightarrow \oplus^m_{j=1}J_j(Z^j)$ in
$C^b_n(\mod\ca)$. Compose these morphisms with the natural injection
$p'' : \oplus^{m-n+1}_{j=1}J_j(P_j) \hookrightarrow
\oplus^m_{j=1}J_j(P_j)$, we obtain a morphism
$$\xymatrix{ r := pp'p'' : \oplus^{m-n+1}_{j=1}J_j(P_j) \ar@{^{(}->}[r] &
\oplus^m_{j=1}J_j(P_j) \ar@{->>}[r] & \oplus^m_{j=1}J_j(Z^j)
\ar@{->>}[r] & Z }$$ in $C^m_n(\mod\ca)$.

By \cite[Proposition 41]{IyaKatMiy13}, $Z$ admits a homotopically
projective resolution $q : Q \rightarrow Z$ with $Q \in C^{\leq
m}_n(\proj\ca)$. Now we check $g:=(\tau_{\geq 1}(q),r) : Y:=
\tau_{\geq 1}(Q) \oplus (\oplus^{m-n+1}_{j=1}J_j(P_j)) \rightarrow
Z$ is a right $C^m_n(\proj\ca)$-approximation of $Z$, where
$\tau_{\geq 1} : C_n(\Mod\ca) \rightarrow C^{\geq 1}_n(\Mod\ca)$ is
the hard truncation functor, i.e., for any $n$-complex $X$ and any
morphism of $n$-complexes $f$, $(\tau_{\geq 1}(X))^i := X^i$ for all
$i \geq 1$ and $(\tau_{\geq 1}(X))^i := 0$ otherwise and
$(\tau_{\geq 1}(f))^i := f^i$ for all $i \geq 1$ and $(\tau_{\geq
1}(f))^i := 0$ otherwise (Ref. \cite[Definition 21]{IyaKatMiy13}).
In another words, for any morphism $g' : Z' \rightarrow Z$ in
$C^m_n(\mod\ca)$ with $Z' \in C^m_n(\proj\ca)$, we need to show that
there is a morphism $g'' : Z' \rightarrow Y$ in $C^m_n(\proj\ca)$
such that $g'=gg''$.

Since $q$ is a quasi-isomorphism, $q^{-1}g' : Z' \rightarrow Q$ is a
morphism in the upper bounded derived category $D^-_n(\mod\ca)$.
Since $Z' \in C^m_n(\proj\ca)$ is homotopically projective, there is
a morphism $h : Z' \rightarrow Q$ in $C^-_n(\proj\ca)$ such that $h
= q^{-1}g'$ in $D^-_n(\mod\ca)$.
$$\xymatrix{ & Z' \ar[d]^-{g'} \ar[dl]_-h \\ Q \ar[r]^-q & Z }$$
Thus $g'=qh$ in $D^-_n(\mod\ca)$, and further in the upper bounded
homotopy category $K^-_n(\mod\ca)$ due to $Z' \in C^m_n(\proj\ca)$.
Hence $g'=qh+l$ in $C^m_n(\mod\ca)$ for some null-homotopy $l : Z'
\rightarrow Z$ in $C^m_n(\mod\ca)$. Since $l$ is a null-homotopy,
$l$ is factored through $p$ (Ref. \cite[Proof of Theorem
19]{IyaKatMiy13}), say $l=pl'$. Since $Z' \in C^m_n(\proj\ca)$, each
component of $Z'$ is projective. Thus $l'$ is factored through $p'$,
say $l'=p'l''$. It is easy to see that a morphism from $Z' \in
C^m_n(\proj\ca)$ to $J_j(P_j)$ must be zero for all $m-n+2 \leq j
\leq m$. Thus $l''$ is factored through $p''$, say $l''=p''l'''$.
Hence $l$ is factored through $r$.
$$\xymatrix{ & & & Z' \ar[dlll]_-{l'''} \ar[dll]^-{l''} \ar[dl]^-{l'} \ar[d]^-l \\
\oplus^{m-n+1}_{j=1}J_j(P_j) \ar@{^{(}->}[r]_-{p''} &
\oplus^m_{j=1}J_j(P_j) \ar@{->>}[r]_-{p'} & \oplus^m_{j=1}J_j(Z^j)
\ar@{->>}[r]_-p & Z }$$ Acting the hard truncation functor
$\tau_{\geq 1}$ on $g'=qh+l$, we get $$g' = \tau_{\geq
1}(q)\tau_{\geq 1}(h)+l = \tau_{\geq 1}(q)\tau_{\geq 1}(h)+rl''' =
(\tau_{\geq 1}(q),r) \left(\begin{array}{c} \tau_{\geq 1}(h)
\\ l'''
\end{array}\right) = gg''$$ where $g'' := \left(\begin{array}{c}
\tau_{\geq 1}(h) \\ l'''
\end{array}\right) : Z' \rightarrow Y = \tau_{\geq 1}(Q) \oplus
(\oplus^{m-n+1}_{j=1}J_j(P_j))$. So $g : Y \rightarrow Z$ is a right
$C^m_n(\proj\ca)$-approximation of $Z$. Therefore, $C^m_n(\proj\ca)$
is contravariantly finite over $C^m_n(\mod\ca)$.

Since $C^m_n(\ca)$ is Krull-Schmidt and $(C^m_n(\ca),\ce^m_n)$ has
right almost split morphisms, $(C^m_n(\ca),\ce^m_n)$ has minimal
right almost split morphisms. Furthermore, one can prove that if $Z
\in C^m_n(\ca)$ is indecomposable and non-$\ce^m_n$-projective then
there is an almost split sequence in $(C^m_n(\ca),\ce^m_n)$ ending
in $Z$.

Applying the obtained result to dualizing $k$-variety $\ca^\op$, we
know that $(C^m_n(\ca^\op),(\ce^\op)^m_n)$ has right almost split
morphisms and if $Z \in C^m_n(\ca^\op)$ is indecomposable and
non-$(\ce^\op)^m_n$-projective then there is an almost split
sequence in $(C^m_n(\ca^\op),(\ce^\op)^m_n)$ ending in $Z$. Since
$(C^m_n(\ca),\ce^m_n) \simeq (C^m_n(\ca^\op), \linebreak
(\ce^\op)^m_n)^\op$, $(C^m_n(\ca),\ce^m_n)$ has left almost split
morphisms and if $X \in C^m_n(\ca)$ is indecomposable and
non-$\ce^m_n$-injective then there is an almost split sequence in
$(C^m_n(\ca),\ce^m_n)$ starting in $X$. Therefore,
$(C^m_n(\ca),\ce^m_n)$ has almost split sequences.
\end{proof}

\subsection{Categories of $n$-cyclic complexes}

Let $\ca$ be an additive category, $n \in \mathbb{N}$ and
$\mathbb{Z}_n = \{0,1, \cdots , n-1\}$ the additive cyclic group of
order $n$. An {\it $n$-cyclic complex} $X$ on $\ca$ is a collection
$(X^i,d^i_X)_{i \in \mathbb{Z}_n}$ with $X^i \in \ca$ and $d^i_X \in
\ca(X^i,X^{i+1})$ such that $d^{i+1}_Xd^i_X = 0$ for all $i \in
\mathbb{Z}_n$ (Ref. \cite[\S 7, Appendix]{PenXia97}). An $n$-cyclic
complex $X = (X^i,d^i_X)_{i \in \mathbb{Z}_n}$ on $\ca$ can be
visualized as the following diagram:
$$\xymatrix{ & & X^1 \ar[rrd]^-{d^1_X} & & \\
X^0 \ar[urr]^-{d^0_X} & & & & X^2 \ar@{.}[dl] \\
& X^{n-1} \ar[ul]_-{d^{n-1}_X} & & X^{n-2} \ar[ll]_-{d^{n-2}_X} &
}$$ A {\it morphism} $f$ from an $n$-cyclic complex $X =
(X^i,d^i_X)_{i \in \mathbb{Z}_n}$ on $\ca$ to an $n$-cyclic complex
$Y = (Y^i,d^i_Y)_{i \in \mathbb{Z}_n}$ on $\ca$ is a collection
$(f^i)_{i \in \mathbb{Z}_n}$ with $f^i \in \ca(X^i,Y^i)$ such that
$f^{i+1}d^i_X=d^i_Yf^i$ for all $i \in \mathbb{Z}_n$, i.e., we have
the following commutative diagram:
$$
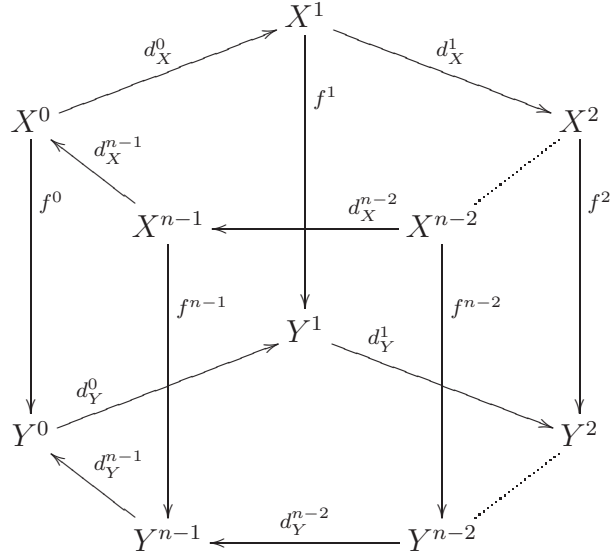
\xymatrix{ & & X^1 \ar[rrd]^-{d^1_X} \ar[ddd]^(.25){f^1} & & \\
X^0 \ar[urr]^-{d^0_X} \ar[ddd]^(.25){f^0} & & & &
X^2 \ar@{.}[dl] \ar[ddd]^(.25){f^2}\\
& X^{n-1} \ar[ul]_-{d^{n-1}_X} \ar[ddd]^(.25){f^{n-1}} & &
X^{n-2} \ar[ll]_(.25){d^{n-2}_X} \ar[ddd]^(.25){f^{n-2}} & \\
& & Y^1 \ar[rrd]^(.25){d^1_Y} & & \\
Y^0 \ar[urr]^(.25){d^0_Y} & & & & Y^2 \ar@{.}[dl] \\
& Y^{n-1} \ar[ul]_-{d^{n-1}_Y} & & Y^{n-2} \ar[ll]_-{d^{n-2}_Y} &
}$$ The {\it composition} of morphisms $f=(f^i)_{i \in \mathbb{Z}_n}
: X \rightarrow Y$ and $g=(g^i)_{i \in \mathbb{Z}_n} : Y \rightarrow
Z$ is $gf := (g^if^i)_{i \in \mathbb{Z}_n} : X \rightarrow Z$. All
$n$-cyclic complexes on $\ca$ and all morphisms between them form
the {\it category of $n$-cyclic complexes on $\ca$}, denoted by
$C_{\mathbb{Z}_n}(\ca)$.

Let $(\ca , \ce)$ be an exact category and $\ce_{\mathbb{Z}_n}$ the
class of composable morphisms $\xymatrix{ X \ar[r]^-{f} & Y
\ar[r]^-{g} & Z }$ such that $\xymatrix{ X^i \ar[r]^-{f^i} & Y^i
\ar[r]^-{g^i} & Z^i }$ is a conflation in $\ce$ for all $i \in
\mathbb{Z}_n$. Then $(C_{\mathbb{Z}_n}(\ca),\ce_{\mathbb{Z}_n})$ is
an exact category.

\begin{corollary} Let $\ca$ be a dualizing $k$-category,
$\cf$ the natural exact structure on the abelian category $\mod\ca$,
and $n \in \mathbb{N}$. Then
$(C_{\mathbb{Z}_n}(\mod\ca),\cf_{\mathbb{Z}_n})$ has almost split
sequences.
\end{corollary}

\begin{proof} Let $Q$ be the quiver with vertices $i$ and arrows
$a_i : i \rightarrow i+1$ for all $i \in \mathbb{Z}_n  = \{0, 1,2,
\cdots , n-1 \}$, i.e.,
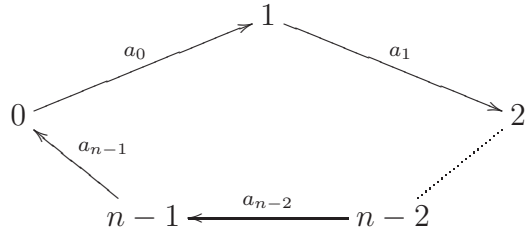
$$\xymatrix{ & & 1 \ar[rrd]^-{a_1} & & \\ 0 \ar[urr]^-{a_0} & & & & 2 \ar@{.}[dl] \\
& n-1 \ar[ul]_-{a_{n-1}} & & n-2 \ar[ll]_-{a_{n-2}} & }$$ $I$ the
ideal of $kQ$ generated by all paths of length $2$, and $\cb :=
kQ/I$. By Proposition~\ref{Proposition-rep-mod} and
Theorem~\ref{Theorem-dualizing}, we know $C_{\mathbb{Z}_n}(\mod\ca)
\simeq \rep(\cb,\mod\ca) \simeq \mod(\cb^\op \otimes_k \ca)$ has
almost split sequences.
\end{proof}

For any $M \in \mod\ca$ and $j \in \mathbb{Z}_n$, denote by $J_j(M)$
the $n$-cyclic complex $X=(X^i,d^i_X)_{i \in \mathbb{Z}_n}$: if $n
\geq 2$ then $X^i := M$ for $i=j,j+1$ and $X^i := 0$ otherwise, and
$d^i_X := 1_M$ for $i=j$ and $d^i_X := 0$ otherwise; if $n=1$ then
$X^0 := M \oplus M$ and $d^0_X := \left(\begin{array}{cc} 0 & 0 \\
1_M & 0 \end{array}\right)$. We say an $n$-cyclic complex
$X=(X^i,d^i_X)_{i \in \mathbb{Z}_n}$ is {\it stalk} if $X^i \neq 0$
for at most one $i \in \mathbb{Z}_n$ and $d^i_X = 0$ for all $i \in
\mathbb{Z}_n$. Denote by $C^b_{\mathbb{Z}_n}(\ca)$ the smallest full
subcategory of $C_{\mathbb{Z}_n}(\ca)$ containing all stalk
$n$-cyclic complexes and closed under finite extensions, and
$\ce^b_{\mathbb{Z}_n} :=
\ce_{\mathbb{Z}_n}|C^b_{\mathbb{Z}_n}(\ca)$. Then
$(C^b_{\mathbb{Z}_n}(\ca),\ce^b_{\mathbb{Z}_n})$ is a full exact
subcategory of $(C_{\mathbb{Z}_n}(\ca),\ce_{\mathbb{Z}_n})$.

\begin{corollary} \label{Corollary-C^b_{Z_n}-ass}
Let $\ca$ be a dualizing $k$-category with trivial exact structure
$\ce$ and $\gl(\mod\ca) < \infty$, and $n \in \mathbb{N}$. Then the
exact category $(C^b_{\mathbb{Z}_n}(\ca),\ce^b_{\mathbb{Z}_n})$ has
almost split sequences.
\end{corollary}

\begin{proof} First of all, we show that
$(C^b_{\mathbb{Z}_n}(\ca),\ce^b_{\mathbb{Z}_n})$ has right almost
split morphisms. We have known $(\ca,\ce) \simeq (\proj\ca,\cf |
\proj\ca)$. Thus $(C^b_{\mathbb{Z}_n}(\ca),\ce^b_{\mathbb{Z}_n})
\simeq (C^b_{\mathbb{Z}_n}(\proj\ca),(\cf |
\proj\ca)^b_{\mathbb{Z}_n})$. So it is enough to prove that the
exact category $(C^b_{\mathbb{Z}_n}(\proj\ca),(\cf |
\proj\ca)^b_{\mathbb{Z}_n}) =
(C^b_{\mathbb{Z}_n}(\proj\ca),\cf_{\mathbb{Z}_n} |
C^b_{\mathbb{Z}_n}(\proj\ca))$ has right almost split morphisms. For
this, it suffices to show that $C^b_{\mathbb{Z}_n}(\proj\ca)$ is a
contravariantly finite subcategory of $C_{\mathbb{Z}_n}(\mod\ca)$
closed under $\cf_{\mathbb{Z}_n}$-extensions and direct summands.
Clearly, we need only to prove that $C^b_{\mathbb{Z}_n}(\proj\ca)$
is contravariantly finite over $C_{\mathbb{Z}_n}(\mod\ca)$, i.e.,
for any $Z \in C_{\mathbb{Z}_n}(\mod\ca)$, there exists a right
$C^b_{\mathbb{Z}_n}(\proj\ca)$-approximation $g : Y \rightarrow Z$.

By \cite[Page 53, Proof of Proposition 7.1]{PenXia97}, there is an
epimorphism $p : \oplus^{n-1}_{j=0}J_j(Z^j) \rightarrow Z$ in
$C_{\mathbb{Z}_n}(\mod\ca)$. Since $\mod\ca$ has enough projective
objects, for any $j \in \mathbb{Z}_n$, we have an epimorphism $P_j
\twoheadrightarrow Z^j$ in $\mod\ca$ with $P^j \in \proj\ca$. Thus
there is an epimorphism $p' : \oplus^{n-1}_{j=0}J_j(P_j) \rightarrow
\oplus^{n-1}_{j=0}J_j(Z^j)$ in $C_{\mathbb{Z}_n}(\mod\ca)$. So we
obtain a morphism
$$\xymatrix{ r=pp' :
\oplus^{n-1}_{j=0}J_j(P_j) \ar@{->>}[r] & \oplus^{n-1}_{j=0}J_j(Z^j)
\ar@{->>}[r] & Z }$$ in $C_{\mathbb{Z}_n}(\mod\ca)$.

From $n$-cyclic complex $Z \in C_{\mathbb{Z}_n}(\mod\ca)$, we can
construct a complex $\tilde{Z}=(\tilde{Z}^i,d^i_{\tilde{Z}})_{i \in
\mathbb{Z}}$ where $\tilde{Z}^i := Z^i$ and $d^i_{\tilde{Z}} :=
d^i_Z$ for all $i \in \mathbb{Z}$. By \cite[Lemma 5.7.2]{Wei94},
$\tilde{Z}$ admits a Cartan-Eilenberg resolution $P_{**}$ whose
total complex Tot$^{\oplus}(P_{**})$, denoted by $\tilde{Q}$, is
quasi-isomorphic to $\tilde{Z}$ in $C_2(\Mod\ca)$. Let $\tilde{q} :
\tilde{Q} \rightarrow \tilde{Z}$ be such a quasi-isomorphism. Due to
$\gl(\mod\ca) < \infty$, we can choose projective resolutions of all
cohomologies and coboundaries, and thus cocycles and components, of
$\tilde{Z}$ during the construction of $P_{**}$ to be of finite
length. Namely, we can assume that each component of $\tilde{Q}$ is
finitely generated projective, i.e., $\tilde{Q} \in C_2(\proj\ca)$.
From $\tilde{Q}$, we can construct an $n$-cyclic complex $Q =
(Q^i,d^i_Q)_{i \in \mathbb{Z}_n}$ where $Q^i := \tilde{Q}^i$ and
$d^i_Q := d^i_{\tilde{Q}}$ for all $i \in \mathbb{Z}_n$. Then $Q \in
C^b_{\mathbb{Z}_n}(\proj\ca)$ and there is a quasi-isomorphism $q :
Q \rightarrow Z$ (cf. \cite[Proposition 2.5]{Zha14} and \cite[Lemma
3.5]{Sta16}). Now we check $g := (q,r) : Y := Q \oplus
(\oplus^{n-1}_{j=0}J_j(P_j)) \rightarrow Z$ is a right
$C^b_{\mathbb{Z}_n}(\proj\ca)$-approximation of $Z$. In another
words, for any morphism $g' : Z' \rightarrow Z$ in
$C_{\mathbb{Z}_n}(\mod\ca)$ with $Z' \in
C^b_{\mathbb{Z}_n}(\proj\ca)$, we need to show that there is a
morphism $g'' : Z' \rightarrow Y$ in $C^b_{\mathbb{Z}_n}(\proj\ca)$
such that $g'=gg''$.

Since $q$ is a quasi-isomorphism, we have a morphism $q^{-1}g' : Z'
\rightarrow Q$ in $D_{\mathbb{Z}_n}(\mod\ca)$. Since $Z' \in
C^b_{\mathbb{Z}_n}(\proj\ca)$ is homotopically projective (Ref.
\cite[Proposition 2.4]{Zha14}), there is a morphism $h : Z'
\rightarrow Q$ in $C^b_{\mathbb{Z}_n}(\proj\ca)$ such that $h =
q^{-1}g'$ in the derived category $D_{\mathbb{Z}_n}(\mod\ca)$.
$$\xymatrix{ & Z' \ar[d]^-{g'} \ar[dl]_-h \\ Q \ar[r]^-q & Z }$$
So $g'=qh$ in $D_{\mathbb{Z}_n}(\mod\ca)$, and further in the
homotopy category $K_{\mathbb{Z}_n}(\mod\ca)$ due to $Z' \in
C^b_{\mathbb{Z}_n}(\proj\ca)$. Hence $g'=qh+l$ in
$C_{\mathbb{Z}_n}(\mod\ca)$ for some null-homotopy $l : Z'
\rightarrow Z$ in $C_{\mathbb{Z}_n}(\mod\ca)$. Since $l$ is a
null-homotopy, $l$ is factored through $p$ (Ref. \cite[Page 53,
Proof of Proposition 7.1]{PenXia97}), say $l=pl'$. Since $Z' \in
C^b_{\mathbb{Z}_n}(\proj\ca)$, each component of $Z'$ is projective.
Thus $l'$ is factored through $p'$, say $l'=p'l''$. Hence $l$ is
factored through $r$.
$$\xymatrix{ & & Z' \ar[dll]_-{l''} \ar[dl]^-{l'} \ar[d]^-l \\
\oplus^{n-1}_{j=0}J_j(P_j) \ar@{->>}[r]_-{p'} &
\oplus^{n-1}_{j=0}J_j(Z^j) \ar@{->>}[r]_-p & Z }$$ Furthermore, $$g'
= qh+l = qh+rl'' = (q,r) \left(\begin{array}{c} h \\ l''
\end{array}\right) = gg''$$ where $g'' := \left(\begin{array}{c} h
\\ l''
\end{array}\right) : Z' \rightarrow Y = Q \oplus (\oplus^{n-1}_{j=0}J_j(P_j))$.
So $g : Y \rightarrow Z$ is a right
$C^b_{\mathbb{Z}_n}(\proj\ca)$-approximation of $Z$. Therefore,
$C^b_{\mathbb{Z}_n}(\proj\ca)$ is contravariantly finite over
$C_{\mathbb{Z}_n}(\mod\ca)$.

Since $C^b_{\mathbb{Z}_n}(\ca)$ is Krull-Schmidt and
$(C^b_{\mathbb{Z}_n}(\ca),\ce^b_{\mathbb{Z}_n})$ has right almost
split morphisms, $(C^b_{\mathbb{Z}_n}(\ca),\ce^b_{\mathbb{Z}_n})$
has minimal right almost split morphisms. Furthermore, one can prove
that if $Z \in C^b_{\mathbb{Z}_n}(\ca)$ is indecomposable and
non-$\ce^b_{\mathbb{Z}_n}$-projective then there is an almost split
sequence in $(C^b_{\mathbb{Z}_n}(\ca),\ce^b_{\mathbb{Z}_n})$ ending
in $Z$.

Applying the obtained result to dualizing $k$-variety $\ca^\op$, we
know that $(C^b_{\mathbb{Z}_n}(\ca^\op),(\ce^\op)^b_{\mathbb{Z}_n})$
has right almost split morphisms and if $Z \in
C^b_{\mathbb{Z}_n}(\ca^\op)$ is indecomposable and
non-$(\ce^\op)^b_{\mathbb{Z}_n}$-projective then there is an almost
split sequence in
$(C^b_{\mathbb{Z}_n}(\ca^\op),(\ce^\op)^b_{\mathbb{Z}_n})$ ending in
$Z$. Since $(C^b_{\mathbb{Z}_n}(\ca),\ce^b_{\mathbb{Z}_n}) \simeq
(C^b_{\mathbb{Z}_n}(\ca^\op), \linebreak
(\ce^\op)^b_{\mathbb{Z}_n})^\op$,
$(C^b_{\mathbb{Z}_n}(\ca),\ce^b_{\mathbb{Z}_n})$ has left almost
split morphisms and if $X \in C^b_{\mathbb{Z}_n}(\ca)$ is
indecomposable and non-$\ce^b_{\mathbb{Z}_n}$-injective then there
is an almost split sequence in
$(C^b_{\mathbb{Z}_n}(\ca),\ce^b_{\mathbb{Z}_n})$ starting in $X$.
Therefore, $(C^b_{\mathbb{Z}_n}(\ca),\ce^b_{\mathbb{Z}_n})$ has
almost split sequences.
\end{proof}

\begin{corollary} Let $\ca$ be a dualizing $k$-category with
trivial exact structure $\ce$ and $\gl(\mod\ca) \leq 1$, and $n \in
\mathbb{N}$. Then the exact category
$(C_{\mathbb{Z}_n}(\ca),\ce_{\mathbb{Z}_n})$ has almost split
sequences.
\end{corollary}

\begin{proof} Analogous to \cite[Lemma 4.2]{Bri13} (cf.
\cite[Proposition 9.7]{Gor13}), we can prove
$C^b_{\mathbb{Z}_n}(\ca) = C_{\mathbb{Z}_n}(\ca)$ due to
$\gl(\mod\ca) \leq 1$. Then this corollary follows from
Corollary~\ref{Corollary-C^b_{Z_n}-ass}.
\end{proof}

\noindent {\footnotesize {\bf ACKNOWLEDGEMENT.} The authors are
sponsored by Project 11571341 NSFC.}

\end{document}